\documentclass[a4paper,leqno]{article}
\usepackage{amssymb}
\usepackage{amsmath}
\usepackage{mathrsfs}
\usepackage{color}
\usepackage{ulem}
\usepackage{comment}
\usepackage{theorem}
\usepackage{ascmac}
\usepackage{accents}
\newcommand{\pref}[1]{(\ref{#1})}
\newcommand{\be}{\begin{equation}}
\newcommand{\ee}{\end{equation}}
\newcommand{\qed}{{\unskip\nobreak\hfil\penalty50\quad\null\nobreak\hfil
	$\square$\parfillskip0pt\finalhyphendemerits0\par\medskip}}

\newcommand{\vecc}{\mbox{\boldmath $ c $}}

\newcommand{\vecf}{\mbox{\boldmath $ f $}}

\newcommand{\veckappa}{\mbox{\boldmath $ \kappa $}}
\newcommand{\vecnu}{\mbox{\boldmath $ \nu $}}
\newcommand{\vecrho}{\mbox{\boldmath $ \rho $}}
\newcommand{\vectau}{\mbox{\boldmath $ \tau $}}

\renewcommand{\epsilon}{\varepsilon}

\newtheorem{thm}{Theorem}[section]
\newtheorem{prop}{Proposition}[section]
\newtheorem{lem}{Lemma}[section]
\newtheorem{cor}{Corollary}[section]

\theorembodyfont{\rmfamily}
\newtheorem{rem}{Remark}[section]
\newtheorem{proof}{\normalfont\itshape Proof.}

\makeatletter
 \def\wideubar{\underaccent{{\cc@style\underline{\mskip10mu}}}}
\makeatother
\title{Asymptotic analysis for non-local curvature flows for plane curves with a general rotation number}
\author{Takeyuki Nagasawa\thanks{Graduate School of Science and Engineering,
Saitama University,
Japan,
e-mail:
tnagasaw@rimath.saitama-u.ac.jp;
Partly supported by Grant-in-Aid for Scientific Research (C) 
(17K05310),
Japan Society for the Promotion Science.}
\and Kohei Nakamura\thanks{Graduate School of Science and Engineering,
Saitama University,
Japan,
e-mail:
knakamura@rimath.saitama-u.ac.jp}}
\date{}
\begin{document}
\maketitle
\begin{abstract}
Several non-local curvature flows for plane curves with a general rotation number are discussed in this work.  
The types of flows include the area-preserving flow and the length-preserving flow.
We have a relatively good understanding of these flows for plane curves with the rotation number one.
In particular,
when the initial curve is strictly convex,
the flow exists globally in time,
and converges to a circle as time tends to infinity.
Even if the initial curve is not strictly convex,
a global solution,
if it exists,
converges to a circle.
Here,
we deal with curves with a general rotation number,
and show,
not only a similar result for global solutions,
but also a blow-up criterion,
upper estimates of the blow-up time,
and blow-up rate from below.
For this purpose,
we use a geometric quantity which has never been considered before.
\\
MSC2020 Mathematics Subject Classification:
53E10,
35K93,
35B40,
53A04
\end{abstract}
\section{Introduction}
\par
In this paper,
we deal with curvature flows comprising non-local terms for plane curves with a general rotation number.
Let $ \vecf $ be an $ \mathbb{R}^2 $-valued function on $ \mathbb{R} / L(t) \mathbb{Z} \times [ 0 , T ) $ such that for a fixed $ t \in [ 0 , T ) $,
it is an arc-length parametrization of a closed plane curve with total length $ L(t) $.
In the following text,
we simply denote $ L(t) $ as $ L $ in many cases.
To explain the curvature flow that is considering in this work,
we introduce a certain geometric quantity.
For a fixed $ t \in [ 0 , T ) $,
$ s \in \mathbb{R} / L \mathbb{Z} $ is an arc-length parameter.
Then,
$ \vectau = \partial_s \vecf $ and $ \veckappa = \partial_s^2 \vecf $ are the unit tangent vector and the curvature vector respectively.
The vector $ \vecnu $ is a unit normal vector given by rotating $\vectau$ counter-clockwise by $ \frac \pi 2 $.
The curvature $ \kappa $ and its deviation $ \tilde \kappa $ are given by
\[
	\kappa = \veckappa \cdot \vecnu
	,
	\quad
	\tilde \kappa = \kappa - \frac 1L \int_0^L \kappa \, ds .
\]
Here,
$ \tilde \kappa $ is a non-local quantity.
The equation we consider is of the following form:
\[
	\partial_t \vecf =
	\left( \tilde \kappa - \frac gL \right) \vecnu .
\]
Here,
we assume that the function $ g $ is a scale-invariant non-local quantity determined by $ \vecf $.
That is,
set
$ \vecf_\lambda (s) = \frac 1 \lambda \vecf ( \lambda s ) $ ($ s \in \mathbb{R} / \lambda^{-1} L \mathbb{Z} $),
then,
\[
	g ( \vecf_\lambda ) = g ( \vecf ) .
\]
\par
Here we study three cases of $ g $:
\begin{itemize}
\item[(AP)]
If we set $ g \equiv 0 $,
then our equation represents the area-preserving flow.
In fact,
we set $ A $ as
\[
	A = - \frac 12 \int_0^L \vecf \cdot \vecnu \, ds
\]
which is the enclosed area when $ \mathrm{Im} \vecf $ is a simple curve.
Consequently,
it holds that
\[
	\frac { dA } { dt } = 0 .
\]
\item[(LP)]
Let $ \displaystyle{ g = L \left( \int_0^L \kappa \, ds \right)^{-1} \int_0^L \tilde \kappa^2 ds } $.
Here,
the equation represents the length-preserving flow:
\[
	\frac { dL } { dt } = 0.
\]
\item[(JP)]
Jiang-Pan considered an equation with $ \displaystyle{ g = \frac { L^2 } { 2A } - \int_0^L \kappa \, ds } $ in \cite{JP}.
Here,
the isoperimetric ratio does not increase along with the flow:
\[
	\frac d { dt } \frac { L^2 } A = - \frac { 2L } A \int_0^L \| \partial_t \vecf \|^2 ds .
\]
\end{itemize}
\par
Let 
\[
	n = \frac 1 { 2 \pi } \int_0^L \kappa \, ds
\]
be the rotation number.
For classical solutions,
the rotation number $ n $ is independent of $ t $.
There are a multitude of literature available considering the case when $ n = 1 $ in the above equations.
First of all,
we should mention Gage's result \cite{G}.
Assume that $ \mathrm{Im} \vecf (0) $ is a strictly convex,
closed curve with a rotation number equal to 1 in the class of $ C^2 $.
Then,
the solution $ \vecf $ with the initial data $ \vecf (0) $ exists globally in time,
and $ \mathrm{Im} \vecf (t) $ converges to a circle with a surrounding area $ A(0) $ as $ t \to \infty $.
Similar results for (LP) and (JP) were proved by \cite{MZ} and \cite{JP} respectively under the convexity condition.
The authors considered flows without the convexity condition in \cite{Nagasawa-Nakamura1,Nakamura2}.
Instead of convexity,
we assume the global existence of the solution.
Then the solution of (AP),
(LP),
or (JP) converges to a circle as $ t \to \infty $ exponentially.
As a result,
the curvature uniformly converges to a positive constant,
and thus,
the curve becomes convex in finite time.
In our previous works,
the isoperimetric deficit
\[
	I_{-1} = 1 - \frac { 4 \pi A } { L^2 }
\]
played an important role.
First,
we show the decay of $ I_{-1} $.
Set
\[
	I_\ell
	=
	L^{ 2 \ell + 1 } \int_0^L \left| \tilde \kappa^{( \ell ) } \right|^2 ds
	\mbox{ for } \ell \in \{ 0 \} \cup \mathbb{N} .
\]
In \cite{Nagasawa-Nakamura1},
we showed the inequality
\be
	I_j
	\leqq
	C \left( I_{-1}^{ \frac { \ell - j } 2 } I_\ell
	+ I_{-1}^{ \frac { \ell - j } { \ell + 1 } }
	I_\ell^{ \frac { j + 1 } { \ell + 1 } }
	\right)
	\label{interpolation1}
\ee
for an integer $ j \in [ 0 , \ell ] $ with a positive constant $ C = C ( j , \ell ) $ independent of the total length of curve.
Since $ I_{-1} $ is small for a sufficiently large $ t $,
we can regard this inequality as an embedding with a small embedding constant.
We showed the exponential decay of $ I_\ell $ using the standard energy method,
combining the above inequality.
Finally,
using the decay of $ I_\ell $,
we showed the convergence of $ \mathrm{Im} \vecf $ to a circle.
\par
In this paper we study the case of $ n > 1 $,
when the isomerimetric deficit is
\[
	I_{-1} = 1 - \frac { 4 n \pi A } { L^2 } .
\]
The isoperimetric inequality shows $ I_{-1}\geqq 0 $ when $ n = 1 $.
However,
$ I_{-1} $ is not necessarily non-negative for $ n > 1 $.
This implies the technique used in \cite{Nagasawa-Nakamura1,Nakamura2} is not applicable for $ n > 1 $.
In spite of this,
$ I_{-1} $ gives us some useful information.
For example,
we can show that if $ I_{-1} $ is negative for $ t = 0 $,
then the solution blows up in finite time.
See our first main result,
Theorem \ref{Theorem blow-up}.
This implies $ I_{-1} \geqq 0 $ for global solutions.
Unfortunately,
the inequality \pref{interpolation1} does not holds for $ n > 1 $.
Therefore,
we use a geometric quantity which has never been considered before,
given as follows:
\[
	\widetilde I_{-1}
	=
	\frac 1L
	\left\|
	\frac { 2 \pi n } L
	\left( \vecf - \frac 1L \int_0^L \vecf \, ds \right)
	+
	\vecnu
	\right\|_{ L^2 }^2 .
\]
Then we can show
\be
	I_j
	\leqq
	C \left( \widetilde{ I }_{-1}^{ \frac { \ell - j } 2 } I_\ell
	+ \widetilde{ I }_{-1}^{ \frac { \ell - j } { \ell + 1 } }
	I_\ell^{ \frac { j + 1 } { \ell + 1 } }
	\right)
	.
	\label{interpolation2}
\ee
Once we obtain this inequality,
we can obtain a similar asymptotic for the global solutions of (AP),
(LP),
and (JP).
See Theorem \ref{Theorem convergence},
which is the second main results of our paper.
\par
We prepare several inequalities and estimates for closed curves with a rotation number $ n $,
in \S~\ref{Preliminaries}.
Using these,
in \S~\ref{Blow-up},
we discuss blow-up solutions with blow-up time estimates,
blow-up quantities,
and blow-up rates.
In the final \S~\ref{Global solutions},
the convergence to an $ n $-fold circle of global solutions is proved.
\begin{rem}
The argument in this paper works even for the case where $ n = 1 $,
and therefore,
the results here include those of the previous papers \cite{Nagasawa-Nakamura1,Nakamura2}.
However,
since the argument there does not need $ \widetilde I_{-1} $,
the authors think that it is simpler than that presented in this paper.
\end{rem}
\section{Preliminaries}
\label{Preliminaries}
\par
In this section,
we provide several estimates and inequalities for plane curves.
Those in \S~\ref{Estimates} hold for curves which are not necessarily solutions of the flows.
We derive the basic properties of flows in \S~\ref{Properties}.
\subsection{Estimates for plane curves}
\label{Estimates}
\par
Let $ \vecf = ( f_1 , f_2 ) $ be an arc-length parametrization of a plane curve with the rotation number $ n \geqq 1 $.
Set
\[
	f = f_1 + i f_2 ,
	\quad
	\nu = \nu_1 + i \nu_2 = - f_2^\prime + i f_1^\prime = i f^\prime
	.
\]
The functions $ \displaystyle{ \varphi_k = \frac 1 { \sqrt { 2 \pi } } \exp \left( \frac { 2 \pi i k s } L \right) } $ for $ k \in \mathbb{Z} $ generate the standard complete orthogonal system of $ L^2 ( \mathbb{R} / L \mathbb{Z} ) $.
Let $ \hat f (k) $ be the Fourier coefficient of $ f $.
Subsequently,
we can derive the following relations in a manner similar to \cite[Corollary 2.1]{Nagasawa-Nakamura1},
where we dealt with the case of $ n = 1 $.
\begin{lem}
\begin{align}
	\sum_{ k \in \mathbb{Z} } k | \hat f(k) |^2
	= & \
	\frac { LA } \pi
	\label{ser1}
	,
	\\
	\sum_{ k \in \mathbb{Z} } k^2 | \hat f(k) |^2
	= & \
	\left( \frac L { 2 \pi } \right)^2
	\int_0^L \kappa^0 ds
	=
	\frac { L^3 } { 4 \pi^2 } ,
	\label{ser2}
	\\
	\sum_{ k \in \mathbb{Z} } k^3 | \hat f(k) |^2
	= & \
	\left( \frac L { 2 \pi } \right)^3
	\int_0^L \kappa \, ds
	=
	\frac { n L^3 } { 4 \pi^2 } ,
	\label{ser3}
	\\
	\sum_{ k \in \mathbb{Z} } k^4 | \hat f(k) |^2
	= & \
	\left( \frac L { 2 \pi } \right)^4
	\int_0^L \kappa^2 ds
	,
	\label{ser4}
	\\
	\sum_{ k \in \mathbb{Z} } k^5 | \hat f(k) |^2
	= & \
	\left( \frac L { 2 \pi } \right)^5
	\int_0^L \kappa^3 ds
	,
	\label{ser5}
	\\
	\sum_{ k \in \mathbb{Z} } k^6 | \hat f(k) |^2
	= & \
	\left( \frac L { 2 \pi } \right)^6
	\int_0^L \left\{ \kappa^4 + ( \kappa^\prime )^2 \right\} ds.
	\label{ser6}
\end{align}
\end{lem}
\par
Note that we have
\be
	\sum_{ k \in \mathbb{Z} } k^2 ( k - n ) | \hat f (k) |^2
	= 0
	\label{useful}
\ee
from \pref{ser2}--\pref{ser3}.
The above is very useful for our analysis.
\begin{lem}
We have
\[
	I_0
	=
	\frac { 16 \pi^4 } { L^3 }
	\sum_{ k \in \mathbb{Z} } k^3 ( k - n ) | \hat f (k) |^2
	=
	\frac { 16 \pi^4 } { L^3 }
	\sum_{ k \in \mathbb{Z} }
	k^2 ( k - n )^2 | \hat f (k) |^2
	.
\]
\label{Lemma I_0}
\end{lem}
\begin{proof}
We obtain the first expression of $ I_0 $ as
\begin{align*}
	I_0 = & \
	L \int_0^L \tilde \kappa^2 ds
	=
	L \int_0^L \tilde \kappa \kappa \, ds
	=
	L \left( \int_0^L \kappa^2 ds - \frac { 2 \pi n } L \int_0^L \kappa \, ds \right)
	\\
	= & \
	\frac { 16 \pi^4 } { L^3 }
	\sum_{ k \in \mathbb{Z} } k^3 ( k - n ) | \hat f (k) |^2
\end{align*}
from \pref{ser4} and \pref{ser3}.
Combining this with \pref{useful},
we obtain the second expression.
\qed
\end{proof}
\par
The non-negativity of $ I_0 $ is not obvious from the first expression.
However,
it can be seen in the second one.
Furthermore,
we see from the second expression that $ I_0 = 0 $ if and only if $ \mathrm{Im} \vecf $ is an $ n $-fold circle.
\par
The isoperimetric inequality holds even if $ n $ is not  $ 1 $.
\begin{lem}
We have $ L^2 - 4 \pi A \geqq 0 $.
\label{isoperimetric inequality}
\end{lem}
\begin{proof}
It follows from \pref{ser2} and \pref{ser1} that
\[
	L^2 - 4 \pi A
	=
	\frac { 4 \pi^2 } L
	\left( \frac { L^3 } { 4 \pi^2 } - \frac { LA } \pi \right)
	=
	\frac { 4 \pi^2 } L
	\sum_{ k \in \mathbb{Z} } k ( k - 1 ) | \hat f (k) |^2
	\geqq 0
	.
\]
\qed
\end{proof}
\par
Similarly,
$ I_{-1} $ has two expressions.
\begin{lem}
We have
\[
	I_{-1}
	=
	\frac { 4 \pi^2 } { L^3 }
	\sum_{ k \in \mathbb{Z} } k ( k - n ) | \hat f (k) |^2
	=
	-
	\frac { 4 \pi^2 } { n L^3 }
	\sum_{ k \in \mathbb{Z} \setminus \{ 0 \} } k ( k - n )^2 | \hat f (k) |^2
	.
\]
\label{Lemma I_{-1}}
\end{lem}
\begin{proof}
It follows from \pref{ser2} and \pref{ser1} that
\[
	I_{-1}
	=
	1 - \frac { 4 \pi n A } { L^2 }
	=
	\frac { 4 \pi^2 } { L^3 }
	\left(
	\frac { L^3 } { 4 \pi^2 } - \frac { n LA } \pi
	\right)
	=
	\frac { 4 \pi^2 } { L^3 }
	\sum_{ k \in \mathbb{Z} } k ( k - n ) | \hat f (k) |^2
	.
\]
The second expression of $ I_{-1} $ is obtained from the above and \pref{useful}.
\qed
\end{proof}
\par
Since $ k ( k - n ) $ is not necessarily non-negative when $ n > 1 $,
we know the same holds for $ I_{-1} $.
However,
the modulus of $ I_{-1} $ can be estimated by $ I_0 $ for $ n \geqq 1 $ as follows.
This is Wirtinger's inequality when $ n = 1 $.
\begin{lem}
It holds that $ 4 \pi^2 n | I_{-1} | \leqq I_0 $.
\label{Wirtinger type}
\end{lem}
\begin{proof}
From Lemmas \ref{Lemma I_0}--\ref{Lemma I_{-1}} we obtain
\begin{align*}
	I_0 \pm 4 \pi^2 n I_{-1}
	= & \
	\frac { 16 \pi^4 } { L^3 }
	\sum_{ k \in \mathbb{Z} }
	\left\{
	k^2 ( k - n )^2 \mp k ( k - n )^2
	\right\} | \hat f (k) |^2
	\\
	= & \
	\frac { 16 \pi^4 } { L^3 }
	\sum_{ k \in \mathbb{Z} }
	k ( k \mp 1 ) ( k - n )^2 | \hat f (k) |^2
	\geqq 0
	.
\end{align*}
Here,
we use $ k ( k \mp 1 ) \geqq 0 $ for $ k \in \mathbb{Z} $.
\qed
\end{proof}
Set
\[
	\widetilde I_{-1} =
	\frac { 4 \pi^2 } { L^3 } \sum_{ k \in \mathbb{Z} \setminus \{ 0 \} } ( k - n )^2 | \hat f (k) |^2
	.
\]
\begin{prop}
We have
\[
	\widetilde I_{-1}
	=
	\frac 1L
	\left\| \frac { 2 \pi n } L \left( \vecf - \frac 1L \int_0^L \vecf \, ds \right) + \vecnu \right\|_{ L^2 }^2
	.
\]
$ \widetilde I_{-1} $ vanishes if and only if $ \mathrm{Im} \vecf $ is an $ n $-fold circle.
\end{prop}
\begin{proof}
Setting
\[
	\tilde f = f - \frac 1L \int_0^L f \, ds ,
\]
we have
\[
	\| \tilde f \|_{ L^2 }^2 = \sum_{ k \in \mathbb{Z} \setminus \{ 0 \} } | \hat f ( k ) |^2 .
\]
The squared $ L^2 $-norm of $ \nu $ is
\[
	\| \nu \|_{ L^2 }^2 = \| f^\prime \|_{ L^2 }^2 = \sum_{ k \in \mathbb{Z} } \left( \frac { 2 \pi k } L \right)^2 | \hat f (k) |^2
	=
	\frac { 4 \pi^2 } { L^2 } \sum_{ k \in \mathbb{Z} \setminus \{ 0 \} } k^2 | \hat f (k) |^2
	.
\]
On the other hand,
we have
\[
	\langle \tilde f , \nu \rangle_{ L^2 }
	=
	\langle \tilde f , i f^\prime \rangle
	= - \sum_{ k \in \mathbb{Z} \setminus \{ 0 \} } \frac { 2 \pi k } L | \hat f ( k ) |^2
	= - \frac { 4 \pi^2 } { L^2 } \sum_{ k \in \mathbb{Z} \setminus \{ 0 \} } \frac { k L } { 2 \pi } | \hat f (k) |^2
	.
\]
Since the last right-hand side expression is a real number,
it holds that
\[
	\frac { 4 \pi^2 } { L^2 } \sum_{ k \in \mathbb{Z} \setminus \{ 0 \} } k | \hat f (k) |^2
	=
	- \frac { 2 \pi } L \Re \langle \tilde f , \nu \rangle_{ L^2 } .
\]
Consequently,
we obtain
\begin{align*}
	\frac { 4 \pi^2 } { L^2 } \sum_{ k \in \mathbb{Z} \setminus \{ 0 \} } ( k - n )^2 | \hat f (k) |^2
	= & \
	\| \nu \|_{ L^2 }^2 + \frac { 4 n \pi } L \Re \langle \tilde f , f^\prime \rangle_{ L^2 } + \left( \frac { 2 \pi n } L \right)^2 \| \tilde f \|_{ L^2 }^2
	\\
	= & \
	\left\| \frac { 2 \pi n } L \tilde f + \nu \right\|_{ L^2 }^2
	\\
	= & \
	\left\| \frac { 2 \pi n } L \left( \vecf - \frac 1L \int_0^L \vecf \, ds \right) + \vecnu \right\|_{ L^2 }^2
	.
\end{align*}
$ \widetilde I_{-1} $ vanishes if and only if
\[
	f = \hat f ( 0 ) \varphi_0 + \hat f (n) \varphi_n .
\]
Hence,
$ \mathrm{Im} \vecf $ is an $ n $-fold circle.
\qed
\end{proof}
\par
An estimate similar to Lemma \ref{Wirtinger type} holds for $ \widetilde I_{-1} $ as well.
\begin{lem}
It holds that $ 4 \pi^2 \widetilde I_{-1} \leqq I_0 $.
\label{Wirtinger type tilde}
\end{lem}
\begin{proof}
Since $ k^2 - 1 \geqq 0 $ for $ k \in \mathbb{Z} \setminus \{ 0 \} $,
we have
\[
	I_0 - 4 \pi^2 \widetilde I_{-1}
	=
	\frac { 16 \pi^4 } { L^3 } \sum_{ k \in \mathbb{Z} \setminus \{ 0 \} } ( k^2 - 1 ) ( k - n )^2 | \hat f (k) |^2
	\geqq
	0
	.
\]
\qed
\end{proof}
\par
The next proposition corresponds to \cite[Theorem 2.2]{Nagasawa-Nakamura1}.
\begin{prop}
It holds that
\[
	I_0
	\leqq
	\widetilde I_{-1}^{ \frac 12 }
	\left[
	\int_0^L
	L^3 \left\{ \kappa^4 + \left( \kappa^\prime \right)^2 \right\} ds
	.
	\right]
\]
\end{prop}
\begin{proof}
It follows from Lemma \ref{Lemma I_0},
Schwarz' inequality,
and \pref{ser6} that
\begin{align*}
	I_0
	= & \
	\frac { 16 \pi^4 } { L^3} \sum_{ k \in \mathbb{Z} \setminus \{ 0 \} }
	k^3 ( k - n ) | \hat f (k) |^2
	\\
	\leqq & \
	\frac { 8 \pi^3 } { L^{ \frac 32 } }
	\left\{ \frac { 4 \pi^2 } { L^3 } \sum_{ k \in \mathbb{Z} \setminus \{ 0 \} }( k - n )^2 | \hat f (k) |^2 \right\}^{ \frac 12 }
	\left\{ \sum_{ k \in \mathbb{Z} \setminus \{ 0 \} } k^6 | \hat f (k) |^2 \right\}^{ \frac 12 }
	\\
	=  & \
	\frac { 8 \pi^3 } { L^{ \frac 32 } }
	\widetilde I_{-1}^{ \frac 12 }
	\left\{ \sum_{ k \in \mathbb{Z} \setminus \{ 0 \} } k^6 | \hat f (k) |^2 \right\}^{ \frac 12 }
	\\
	= & \
	\widetilde I_{-1}^{ \frac 12 }
	\left[
	\int_0^L
	L^3 \left\{ \kappa^4 + \left( \kappa^\prime \right)^2 \right\} ds
	\right]
	.
\end{align*}
\qed
\end{proof}
\par
Using this proposition,
we can prove the following estimates in a manner similar to the proof of \cite[Theorem 3.1]{Nagasawa-Nakamura1}.
\begin{thm}
Let $ j \in [ 0 , \ell ] $ be an integer.
Then,
there exists a positive constant $ C = C ( j , \ell ) $ independent of $ L $ such that
\[
	I_j
	\leqq
	C \left( \widetilde{ I }_{-1}^{ \frac { \ell - j } 2 } I_\ell
	+ \widetilde{ I }_{-1}^{ \frac { \ell - j } { \ell + 1 } }
	I_\ell^{ \frac { j + 1 } { \ell + 1 } }
	\right)
	.
\]
\label{Theorem interpolation}
\end{thm}
\subsection{Estimates for flows}
\label{Properties}
\par
In this subsection,
we derive the basic properties of the flows,
which we use in following sections.
Let $ \vecf $ be
a classical solution of one of
(AP),
(LP),
or (JP) on $ [ 0 , T ) $,
and let $ T $ be the maximum existence time.
Since $ \displaystyle{ \frac { dL } { dt } = - \int_0^L \partial_t \vecf \cdot \veckappa \, ds } $,
we have
\[
	\frac { dL^2 } { dt }
	=
	- 2 L \int_0^L \left( \tilde \kappa - \frac gL \right) \kappa \, ds
	=
	- 2 L \int_0^L \tilde \kappa^2 ds
	+
	4 \pi n g
	,
\]
that is,
\be
	\frac { d L^2 } { dt } + 2 I_0 = 4 \pi n g
	.
	\label{ d L^2 / dt }
\ee
Similarly,
we have
\be
	\frac { dA } { dt }
	=
	- \int_0^L \partial_t \vecf \cdot \vecnu \, ds
	=
	- \int_0^L \left( \tilde \kappa - \frac gL \right) ds
	=
	g
	.
	\label{ dA / dt }
\ee
It follows from the above that
\be
	\frac d { dt } \left( L^2 I_{-1} \right) + 2 I_0
	=
	\frac d { dt } \left( L^2 - 4 \pi n A \right) + 2 I_0
	= 
	0 .
	\label{ d L^2 I_{-1} / dt }
\ee
\par
From these,
we summarize the basic properties of each solution as follows.
\begin{prop}
Assume that the initial curve is smooth,
and that $ A(0) $ is positive.
Let $ \vecf $ be
a classical solution of one of
(AP),
(LP),
or (JP) on $ [ 0 , T ) $
and let $ T $ be the maximum existence time.
Then,
the following holds for $ t \in ( 0 , T ) $.
\begin{enumerate}
\item	For solutions of (AP),
\[
	\frac { dA } { dt } = 0 ,
	\quad
	A \equiv A(0) > 0 ,
	\quad
	\frac { dL^2 } { dt } \leqq 0 ,
	\quad
	\frac { d I_{-1} } { dt } \leqq 0 .
\]
\item	For solutions of (LP),
\[
	\frac { dA } { dt } \geqq 0 ,
	\quad
	A \geqq A(0) > 0 ,
	\quad
	\frac { dL^2 } { dt } = 0 ,
	\quad
	\frac { d I_{-1} } { dt } \leqq 0 .
\]
\item	For solutions of (JP),
\[
	A > 0 ,
	\quad
	\frac { d I_{-1} } { dt } \leqq 0 .
\]
\item	For solutions of (AP),
	(LP), (JP),
	\[
		1 - n \leqq I_{-1} \leqq I_{-1} (0) .
	\]
	In other words,
	\[
		4 \pi \leqq \frac { L^2 } A \leqq \frac { L(0)^2 } { A(0) } .
	\]
\end{enumerate}
\label{basic properties}
\end{prop}
\begin{proof}
In the cases of (AP) and (LP),
the signs of $ \displaystyle{ \frac { dA } { dt } } $ and $ \displaystyle{ \frac { d L^2 } { dt } } $ immediately follow from \pref{ dA / dt } and \pref{ d L^2 / dt }.
Therefore,
$ A > 0 $ and
\[
	\frac { d I_{-1} } { dt }
	=
	- \frac d { dt } \frac { 4 \pi n A } { L^2 }
	=
	- \frac { 4 \pi n } { L^2 } \frac { dA } { dt }
	+ \frac {  4\pi n A } { L^4 } \frac { d L^2 } { dt }
	\leqq 0
	.
\]
In the case of (JP),
we prove the positivity of $ A $ by applying the contraction argument.
In this case,
\[
	g = \frac { L^2 I_{-1} } { 2A } .
\]
It follows from \pref{ dA / dt } that
\[
	\frac { d A^2 } { dt }
	=
	2 A g
	=
	L^2 I_{-1}
	.
\]
Assume that $ A( t_0 )^2 = 0 $ for some $ t_0 \in ( 0 , T ) $.
Since $ A^2 \geqq 0 $,
we have
\[
	\frac { d A^2 } { dt } ( t_0 ) = 0 .
\]
Since $ A ( 0 )^2 > 0 $,
there exists $ t_1 \in ( 0 , t_0 ) $ such that
\[
	\frac { d A^2 } { dt } ( t_1 ) < 0 .
\]
This contradicts \pref{ d L^2 I_{-1} / dt }:
\[
	\frac { d^2 A^2 } { dt^2 }
	=
	\frac d { dt } \left( L^2 I_{-1} \right)
	=
	- 2 I_0 \leqq 0
	.
\]
Hence,
$ A > 0 $ on $ ( 0,T ) $.
\pref{ d L^2 I_{-1} / dt } and \pref{ d L^2 / dt } yield
\begin{align*}
	L^2 \frac { d I_{-1} } { dt }
	= & \
	- I_{-1} \frac { d L^2 } { dt }
	- 2 I_0
	=
	- I_{-1} \left( 4 \pi n g - 2 I_0 \right) 
	- 2 I_0
	\\
	= & \
	- 4 \pi n \left( \frac { 2g I_{-1} } { L^2 } + I_0 A \right)
	=
	- 4 \pi n \left( \frac { I_{-1}^2 } { 2 L^4 A } + I_0 A \right)
	\leqq 0
	.
\end{align*}
Since $ I_{-1} $ is non-increasing,
we have $ I_{-1} \leqq I_{-1} (0) $.
Lemma \ref{isoperimetric inequality} gives us
\[
	I_{-1} = 1 - \frac { 4 \pi n A } { L^2 }
	=
	1 - n + n \left( 1 - \frac { 4 \pi A } { L^2 } \right)
	\geqq
	1 - n .
\]
\qed
\end{proof}
\section{Blow-up solutions}
\label{Blow-up}
\par
The non-positivity of $ I_{-1} (0) $ implies that the blow-up phenomena occurs in finite time.
\begin{thm}
Let $ \vecf $ be
a classical solution of one of
(AP),
(LP),
or (JP) on $ [ 0 , T ) $
and let $ T $ be the maximum existence time.
Assume that the initial curve is smooth,
and satisfies $ A(0) > 0 $,
$ I_{-1} (0) < 0 $.
Then,
the solution blows up in finite time.
The blow-up time $ T $ is estimated from above as follows: 
\begin{description}
\item[{\rm (AP)}]
	$ \displaystyle{ T \leqq
	\frac { L(0)^2 - 4 \pi A (0) } { - 8 \pi^2 n I_{-1} (0) } } $,
\item[{\rm (LP)}]
	$ \displaystyle{ T \leqq
	\frac { L(0)^2 - 4 \pi A (0) } { - 8 \pi^2 I_{-1} (0) } } $,
\item[{\rm (JP)}]
	$ \displaystyle{ T \leqq
	\frac { L(0)^2 } { - 8 \pi^2 nI_{-1} (0) } } $.
\end{description}
\label{Theorem blow-up}
\end{thm}
\begin{proof}
In the case of (AP),
$ g \equiv 0 $.
It follows from Proposition \ref{basic properties} that $ I_{-1} (t) \leqq I_{-1} (0) < 0 $.
By \pref{ d L^2 / dt } and Lemma \ref{Wirtinger type},
we have
\[
	\frac { dL^2 } { dt }
	=
	- 2 I_0 (t)
	\leqq
	8 \pi^2 n I_{-1} (t)
	\leqq
	8 \pi^2 n I_{-1} (0)
	.
\]
Integrating this from $ 0 $ to $ t \in ( 0 , T ) $,
and using Lemma \ref{isoperimetric inequality},
we obtain
\[
	4 \pi A (0) - L^2 (0)
	=
	4 \pi A (t) - L^2 (0)
	\leqq
	L^2 (t) - L^2 (0)
	\leqq
	8 \pi^4 n I_{-1} (0) t
	.
\]
Consequently,
$ t $ must satisfy
\[
	t \leqq \frac { L(0)^2 - 4 \pi A(0) } { - 8 \pi^2 n I_{-1} (0) } .
\]
\par
In the case of (LP),
$ \displaystyle{ g = \frac { I_0 } { 2 \pi n } \geqq 0 } $.
Proposition \ref{basic properties} shows $ I_{-1} (t) \leqq I_{-1} (0) < 0 $.
From \pref{ dA / dt } and Lemma \ref{Wirtinger type},
we have
\[
	- \frac { dA } { dt }
	=
	- \frac 1 { 2 \pi n } I_0 (t)
	\leqq
	2 \pi I_{-1} (t)
	\leqq
	2 \pi I_{-1} (0)
	.
\]
We integrate this from $ 0 $ to $ t \in ( 0 , T ) $.
Using Lemma \ref{isoperimetric inequality},
we obtain
\[
	4 \pi A(0) - L(0)^2
	=
	4 \pi A(0) - L(t)^2
	\leqq
	4 \pi ( A(0) - A(t) )
	\leqq
	8 \pi I_{-1} (0) t
	.
\]
Consequently,
$ t $ must satisfy
\[
	t \leqq \frac { L(0)^2 - 4 \pi A(0) } { - 8 \pi^2 I_{-1} (0) } .
\]
\par
In the case of (JP),
$ \displaystyle{ g = \frac { L^2 I_{-1} } { 2A } } $.
It follows from \pref{ d L^2 / dt },
Proposition \ref{basic properties},
and Lemma \ref{Wirtinger type} that
\[
	\frac { d L^2 } { dt }
	= - 2 I_0 (t) + \frac { 2 \pi n L (t)^2 } { A(t) } I_{-1} (t)
	\leqq
	- 2 I_0 (t)
	\leqq
	8 \pi^2 n I_{-1} (t)
	\leqq
	8 \pi^2 n I_{-1} (0)
	.
\]
We integrate this from $ 0 $ to $ t \in ( 0 , T ) $.
Using Lemma \ref{isoperimetric inequality},
we obtain
\[
	- L (0)^2
	\leqq
	L(t)^2 - L(0)^2
	\leqq
	8 \pi^2 n I_{-1} (0) t .
\]
Consequently $ t $ must satisfy
\[
	t \leqq \frac { L(0)^2 } { - 8 \pi^2 n I_{-1} (0) } .
\]
\qed
\end{proof}
\begin{cor}
Let $ \vecf $ be a classical solution of one of
(AP),
(LP),
or (JP) on $ [ 0 , T ) $
and let $ T $ be the maximum existence time.
Assume that the initial curve is smooth,
and that satisfies $ A(0) > 0 $,
and $ I_{-1} (0) = 0 $,
but it is not an $ n $-fold circle.
Then,
$ T < \infty $.
\end{cor}
\begin{proof}
Assume $ T = \infty $.
Then,
Theorem \ref{Theorem blow-up} implies that $ I_{-1} (t) \geqq 0 $ for all $ t \in [ 0 , \infty ) $.
On the other hand,
\pref{ d L^2 I_{-1} / dt } with $ I_{-1} ( 0 ) = 0 $ shows that $ I_{-1} (t) \leqq 0 $.
Hence,
$ I_{-1} ( t ) \equiv 0 $.
When $ t > 0 $,
\[
	\int_0^L \tilde \kappa^2 ds
	=
	\frac { I_0 } L = - \frac 1 { 2L } \frac d { dt } \left( L^2 I_{-1} \right)
	= 0 .
\]
Combining this with the rotation number $ n $,
we find that $ \mathrm{Im} \vecf (t) $ is an $ n $-fold circle.
However,
this does not satisfy the initial condition.
\qed
\end{proof}
\begin{cor}
$ \vecf $ is a classical stationary solution of one of
(AP),
(LP),
or (JP),
if and only of it is an $ n $-fold circle.
\label{stationary solution}
\end{cor}
\begin{proof}
Assume that $ \mathrm{Im} \vecf $ is an $ n $-fold circle.
Then,
$ \tilde kappa \equiv 0 $.
Since $ f = \hat f (0) \varphi_0 + \hat f (n) \varphi_n $,
we see $ I_0 = I_{-1} = 0 $ by Lemmas \ref{Lemma I_0} and \ref{Lemma I_{-1}}.
Hence,
$ \displaystyle{ \tilde \kappa - \frac gL \equiv 0 } $ for each case.
Consequently,
it is a stationary solution.
\par
Conversely,
assume that $ \vecf $ is a stationary solution.
It follows from \pref{ d L^2 I_{-1} / dt } that $ I_0 (t) \equiv 0 $.
Hence,
we can conclude that $ \mathrm{Im} \vecf (t) $ is an $ n $-fold circle in a manner similar to the proof of the previous corollary.
\qed
\end{proof}
\par
Let $ \vecf $ blow up at $ T \in ( 0 , \infty ) $.
Then,
we have
\[
	\limsup_{ t \to T - 0 } I_0 (t) = \infty .
\]
Indeed,
if $ \displaystyle{ \limsup_{ t \to T - 0 } I_0 (t) < \infty } $,
then $ \displaystyle{ \sup_{ t \in ( 0 , T ) } I_0 ( t ) } $ is bounded.
We can show the boundedness of $ \displaystyle{ \sup_{ t \in ( 0 , T ) } I_\ell ( t ) } $ by the standard energy method.
Using this and the equation of the flow,
we can see that $ \vecf (t) $ converges to a smooth function as $ t \to T - 0 $.
Consequently,
the solution can expand beyond $ T $.
This is a contradiction.
\par
Set
\[
	W = \int_0^L \kappa^2 ds .
\]
We will show the blow-up of $ W $ and its blow-up rate.
Firstly,
we consider the limit supremum of $ W $.
\begin{lem}
It holds that $ \displaystyle{ \limsup_{ t \to T - 0 } W(t) = \infty } $.
\label{Lemma 1st}
\end{lem}
\begin{proof}
We have
\[
	LW
	= L \int_0^L \left\{ \tilde \kappa^2 + \left( \frac RL \right)^2 \right\} ds
	= I_0 + R^2
	.
\]
Hence,
\[
	\limsup_{ t \to T - 0 } L(t) W (t) = \infty .
\]
Therefore,
the assertion immediately follows in the case of (LP).
\par
In the case of (AP),
$ L $ is non-increasing by Proposition \ref{basic properties}.
Lemma \ref{isoperimetric inequality} implies that $ L \geqq \sqrt{ 4 \pi A } = \sqrt{ 4 \pi A_0 } $.
Consequently,
$ L(t) $ converges to a positive constant as $ t \to T - 0 $,
and the assertion follows.
\par
We show that $ L(t) $ converges to a positive constant in the case of (JP) as well.
We assume that $ \displaystyle{ \liminf_{ t \to T - 0 } A (t) = 0 } $.
$ I_{-1} $ is monotone by Proposition \ref{basic properties}.
Therefore,
it follows from
\[
	\frac { dA } { dt } = \frac { L^2 } { 2A } I_{-1}
\]
that $ A $ does not oscillate near $ t = T $.
Hence,
we may assume $ \displaystyle{ \lim_{ t \to T - 0 } A (t) = 0 } $.
From the above relation and Proposition \ref{basic properties},
we find that $ \displaystyle{ \frac { dA } { dt } } $ is bounded.
Consequently,
the estimate
\[
	0 < A(t) \leqq C ( T - t )
\]
holds.
Thus,
we have
\[
	0 \leqq
	\frac { A(t)^2 } { T - t }
	\leqq
	\frac { C ( T - t )^2 } { T - t } \to 0 \mbox{ as } t \to T - 0
	,
\]
and therefore,
\[
	\lim_{ t \to T - 0 }
	\frac { A(T-0)^2 - A(t)^2 } { T - t } = 0
	.
\]
This implies that
\[
	\frac { d A^2 } { dt } ( T - 0 ) = 0 .
\]
However,
$ A(0)^2 > 0 $ and $ A( T - 0 )^2 = 0 $ show the existence of $ t_\ast \in ( 0 , T ) $ such that
\[
	\frac { d A^2 } { dt } ( t_\ast ) < 0 .
\]
This contradicts the following:
\[
	\frac { d^2 A^2 } { dt^2 } = - 2 I_0 \leqq 0 .
\]
Now,
we prove $ \displaystyle{ \liminf_{ t \to \infty } A( t ) > 0 } $.
Since
\[
	\frac { d A } { dt } = \frac { L^2 } { 2A } I_{-1}
\]
has a constant sign near $ T - 0 $,
we conclude that
$ \displaystyle{ \lim_{ t \to \infty } A( t ) > 0 } $.
The convergence of $ \displaystyle{ \lim_{ t \to \infty } L( t ) } $ follows from the convergence of $ A $,
and the monotonicity and boundedness of $ I_{-1} $.
Since $ \displaystyle{ \frac { L^2 } A } $ is strictly positive by Proposition \ref{basic properties},
the limit of $ L $ is positive.
\qed
\end{proof}
\par
Next,
we derive the time derivative of $ W $.
Set
\[
	R = \int_0^L \kappa \, ds ,
	\quad
	J_p =
	L^{ p-1 } \int_0^L \tilde \kappa^p ds
	\quad ( p \in \mathbb{N} \setminus \{ 1 \} )
	,
\]
which are scale-invariant quantities.
Note that $ I_0 = J_2 $.
\begin{lem}
It holds that
\[
	\frac { d W } { dt }
	=
	\frac 1 { L^3 }
	\left\{
	- 2 I_1 + J_4 + ( 3R - g ) J_3
	+ 3R ( R - g ) J_2
	-R^3 g
	\right\}
	.
\]
\label{Lemma 2nd}
\end{lem}
\begin{proof}
The proof is a direct calculation:
\begin{align*}
	\frac { dW } { dt }
	= & \
	\int_0^L
	\partial_t \vecf \cdot \left( 2 \partial_s^2 \kappa + \kappa^3 \right) ds
	=
	\int_0^L
	\left( \tilde \kappa - \frac gL \right)
	\left( 2 \partial_s^2 \kappa + \kappa^3 \right)
	ds
	\\
	= & \
	- 2 \int_0^L \left( \partial_s^2 \tilde \kappa \right)^2 ds
	+
	\int_0^L
	\left( \tilde \kappa - \frac gL \right)
	\left( \tilde \kappa + \frac RL \right)^3
	ds
	\\
	= & \
	- \frac { 2 I_1 } { L^3 }
	+
	\int_0^L
	\left(
	\tilde \kappa^3
	+
	\frac { 3 R \tilde \kappa^2 } L
	+
	\frac { 3 R^2 \tilde \kappa } { L^2 }
	+
	\frac { R^3 } { L^3 }
	\right)
	\left( \tilde \kappa - \frac gL \right)
	ds
	\\
	= & \
	- \frac { 2 I_1 } { L^3 }
	+
	\int_0^L
	\left\{
	\tilde \kappa^4
	+
	\left( \frac { 3R } L - \frac gL \right) \tilde \kappa^3
	+
	\left( \frac { 3 R^2 } { L^2 } - \frac { 3R g } { L^2 }
	\right) \tilde \kappa^2
	-
	\frac { R^3 g } { L^4 }
	\right\}
	ds
	\\
	= & \
	\frac 1 { L^3 }
	\left\{
	- 2 I_1 + J_4 + ( 3R - g ) J_3 + 3R ( R - g ) J_2 - R^3 g
	\right\}
	.
\end{align*}
\qed
\end{proof}
\par
Thirdly,
we estimate $ \frac { dW } { dt } $ from above.
\begin{lem}
We have
\[
	\frac { dW } { dt }
	\leqq
	\frac { W^3 } { 2M^2 }
	.
\]
Here,
\[
	M
	=
	\left\{
	\begin{array}{ll}
	\displaystyle{
	C
	}
	& \mbox{for (AP) and (LP)} ,
	\\
	\displaystyle{
	C \left\{
	1
	+
	\left( \frac { L_0^2 } { A_0 } \right)^{ \frac 43 }
	\right\}^{ - \frac 12 }
	}
	\quad
	& \mbox{for (JP)}
	\end{array}
	\right.
\]
with the constant $ C $ being independent of the initial curve and the rotation number.
\label{Lemma 3rd}
\end{lem}
\begin{proof}
Here,
we use Lemma \ref{Lemma 2nd}.
In the case of (AP),
since $ g = 0 $,
we have
\[
	\frac { dW } { dt }
	+
	\frac { 2 I_1 } { L^3 }
	=
	\frac 1 { L^3 } \left( J_4 + 3R J_3 + 3 R^2 J_2 \right)
	.
\]
Set $ \theta = \frac 12 - \frac 1p $.
Then,
Gagliardo-Nirenberg's inequality yields
\[
	| J_p |
	\leqq
	C \left( I_0^{ 1 - \theta } I_1^\theta \right)^{ \frac p2 }
	=
	C I_0^{ \frac p4 + \frac 12 } I_1^{ \frac p4 - \frac 12 }
	.
\]
Hence,
\begin{align*}
	\frac { dW } { dt }
	+ \frac { 2 I_1 } { L^3 }
	\leqq & \
	\frac C { L^3 }
	\left(
	I_0^{ \frac 32 } I_1^{ \frac 12 }
	+
	R I_0^{ \frac 54 } I_1^{ \frac 14 }
	+
	R^2 I_0
	\right)
	\\
	\leqq & \
	\frac { I_1 } { L^3 }
	+
	\frac C { L^3 }
	\left(
	I_0^3 + R^{ \frac 43 } I_0^{ \frac 53 }+ R^2 I_0
	\right)
	.
\end{align*}
Owing of $ 0 \leqq I_0 \leqq LW $ and $ R^2 \leqq LW $,
we obtain
\begin{gather*}
	I_0^3 \leqq L^3 W^3
	,
	\quad
	I_0^{ \frac 53 } \leqq L^{ \frac 53 }W^{ \frac 53 }
	=
	( LW )^{ - \frac 43 } L^3 W^3
	\leqq
	R^{ - \frac 83 } L^3 W^3
	,
	\\
	I_0 \leqq LW = ( LW )^{ -4 } L^3 W^3 \leqq R^{ -8 } L^3 W^3
	.
\end{gather*}
Furthermore,
\[
	R = 2 \pi n \geqq 2 \pi .
\]
Consequently,
we conclude that
\[
	\frac { dW } { dt }
	\leqq
	C
	\left(
	1
	+
	R^{ - \frac 43 }
	+
	R^{ -6 }
	\right) W^3
	\leqq
	C W^3
	.
\]
\par
In the case of (LP),
since $ \displaystyle{ g = \frac { I_0 } R } $,
we have
\begin{align*}
	\frac { dW } { dt }
	+ \frac 1 { L^3 }
	\left( 2 I_1 + 3 I_0^2 + R^2 I_0 \right)
	= & \
	\frac 1 { L^3 }
	\left\{
	J_4 + \left( 3R - \frac { I_0 } R \right) J_3 + 3R^2 I_0
	\right\}
	\\
	\leqq & \
	\frac C { L^3 }
	\left(
	I_0^{ \frac 32 } I_1^{ \frac 12 }
	+
	R I_0^{ \frac 54 } I_1^{ \frac 14 }
	+
	R^{-1} I_0^{ \frac 94 } I_1^{ \frac 14 }
	+
	R^2 I_0
	\right)
	\\
	\leqq & \
	\frac { I_1 } { L^3 }
	+
	\frac C { L^3 }
	\left(
	I_0^3
	+
	R^{ \frac 43 } I_0^{ \frac 53 }
	+
	R^{ - \frac 43 } I_0^3
	+
	R^2 I_0
	\right)
	\\
	\leqq & \
	\frac { I_1 } { L^3 }
	+
	C
	\left(
	1
	+
	R^{ - \frac 43 }
	+
	R^{ -6 }
	\right) W^3
	\\
	\leqq & \
	\frac { I_1 } { L^3 }
	+
	C W^3
	.
\end{align*}
\par
In the case of (JP),
since $ \displaystyle{ g = \frac { L^2 } { 2A } - R } $,
we have
\begin{align*}
	&
	\frac { dW } { dt }
	+
	\frac 1 { L^3 }
	\left(
	2 I_1
	+
	\frac { 3R L^2 } { 2A } I_0
	+
	\frac { R^3 L^2 } { 2A }
	\right)
	\\
	& \quad
	=
	\frac 1 { L^3 }
	\left\{
	J_4 + \left( 3R - \frac { L^2 } { 2A } + R \right) J_3
	+ 6R^2 J_2
	+ R^4
	\right\}
	\\
	& \quad
	\leqq
	\frac C { L^3 }
	\left\{
	I_0^{ \frac 32 } I_1^{ \frac 12 }
	+
	R I_0^{ \frac 54 } I_1^{ \frac 14 }
	+
	\frac { L^2 } A I_0^{ \frac 54 } I_1^{ \frac 14 }
	+
	R^2 I_0
	+
	R^{-2} ( LW )^3
	\right\}
	\\
	& \quad
	\leqq
	\frac { I_1 } { L^3 }
	+
	\frac C { L^3 }
	\left[
	I_0^3
	+
	\left\{ R + \left( \frac { L^2 } A \right) \right\}^{ \frac 43 } I_0^{ \frac 53 }
	+
	R^2 I_0
	+
	R^{-2} L^3 W^3
	\right]
	\\
	& \quad
	\leqq
	\frac { I_1 } { L^3 }
	+
	C
	\left[
	1
	+
	\left\{ \frac 1 R + \left( \frac { L^2 } { R^2 A } \right) \right\}^{ \frac 43 }
	+
	R^{-6}
	+
	R^{-2}
	\right]
	W^3
	\\
	& \quad
	\leqq
	\frac { I_1 } { L^3 }
	+
	C
	\left\{
	1
	+
	\left( \frac { L^2 } { A } \right)^{ \frac 43 }
	\right\}
	W^3
	.
\end{align*}
By Proposition \ref{basic properties},
we have
\[
	\left( \frac { L^2 } A \right)^{ \frac 43 }
	\leqq
	\left( \frac { L_0^2 } { A_0 } \right)^{ \frac 43 }
	.
\]
Consequently,
we can conclude that
\[
	\frac { dW } { dt }
	\leqq
	C
	\left\{
	1
	+
	\left( \frac { L_0^2 } { A_0 } \right)^{ \frac 43 }
	\right\}
	W^3
	.
\]
\qed
\end{proof}
\par
Now,
we prove the following theorem.
\begin{thm}
Let $ T \in ( 0 , T ) $ be the blow-up time for a solution of one of (AP),
(LP),
or (JP).
Then,
$ W(t) $ blows up as
\[
	W (t) \geqq \frac M { \sqrt{ T - t } } ,
\]
where
\[
	M
	=
	\left\{
	\begin{array}{ll}
	\displaystyle{
	C
	}
	& \mbox{for (AP) and (LP)} ,
	\\
	\displaystyle{
	C \left\{
	1
	+
	\left( \frac { L_0^2 } { A_0 } \right)^{ \frac 43 }
	\right\}^{ - \frac 12 }
	}
	\quad
	& \mbox{for (JP)}
	\end{array}
	\right.
\]
with a constant $ C $ that is independent of the initial curve and the rotation number.
\label{Theorem blow up W}
\end{thm}
\begin{proof}
It follows from Lemma \ref{Lemma 3rd} that
\[
	\frac d { dt } W^{-2}
	\geqq
	- M^{-2}
	.
\]
Due to Lemma \ref{Lemma 1st},
there exists a sequence $ \{ t_n \} $ such that $ t_n \to T - 0 $ and $ W( t_n )^{-2} \to 0 $ as $ n \to \infty $.
Integrating the differential inequality from $ t $ to $ t_n $,
we have
\[
	W(t)^{ -2 } - W ( t_n )^{ -2 }
	\leqq
	M^{-2} ( t_n - t )
	.
\]
Therefore,
we obtain the theorem as $ n \to \infty $.
\qed
\end{proof}
\par
The curve $ \mathrm{Im} \vecf $ may have several loops.
When the orientation of a loop is counter-clockwise as $ s $ increases,
it is called a {\it positive} loop.
Otherwise,
it is called a {\it negative} loop.
It has already been shown that $ L(t) $ converges to a positive constant as $ t \to \infty $.
Therefore,
from the above theorem we know that
\[
	\lim_{ t \to T - 0 }
	\max_{ s \in \mathbb{R} / L(t) \mathbb{Z} } \kappa ( s , t )
	= \infty
\]
or
\[
	\lim_{ t \to T - 0 }
	\min_{ s \in \mathbb{R} / L(t) \mathbb{Z} } \kappa ( s , t )
	= - \infty
	.
\]
If a positive/negative loop of $ \mathrm{Im} \vecf $ shrinks as $ t \to T - 0 $,
the maximum/minimum value of the curvature may not remain bounded.
On the other hand,
there is a possibility of the maximum or minimum remaining bounded as $ t \to T - 0 $.
For example,
if a negative loop shrinks as $ T \to \infty $ before the positive loops shrink,
the minimum value of the curvature goes to $ - \infty $,
but the maximum remains bounded.
In the last part of this section,
we discuss the blow-up of the maximum and minimum.
\begin{thm}
Let $ T \in ( 0 , \infty ) $ be the blow-up time for a solution of one of (AP),
(LP),
or (JP).
Assume that
\[
	\limsup_{ t \to T - 0 } \max_{ s \in \mathbb{R} / L(t) \mathbb{Z} } \kappa ( s , t )
	= \infty ,
\]
then it satisfies
\[
	\max_{ s \in \mathbb{R} / L(t) \mathbb{Z} } \kappa ( s , t )
	\geqq
	\frac 1 { \sqrt { 2 ( T - t ) } }
	.
\]
\end{thm}
\begin{proof}
Set
\begin{align*}
	K (t) = & \
	\max_{ s \in \mathbb{R} / L(t) \mathbb{Z} } \kappa ( s , t )
	,
	\\
	\frac { d^+ K } { dt } (t)
	= & \
	\limsup_{ h \to + 0 } \frac { K( t + h ) - K(t) } h
	.
\end{align*}
Define the set $ S_t $ by $ S_t = \{ s \in \mathbb{R} / L(t) \mathbb{Z} \, | \, \kappa ( s,t ) = K(t) \} $.
After re-parametrizing $ \vecf ( \cdot , t ) $ by a new parameter that is independent of $ t $,
we apply \cite[Lemma B.40]{CLN}.
Consequently ,
we can conclude that $ K $ is a continuous function of $ t $,
and that
\[
	\frac { d^+ K } { dt } (t)
	=
	\max_{ s \in S_t } \partial_t \kappa (s,t) .
\]
$ \kappa $ satisfies the equation
\[
	\partial_t \kappa
	=
	\partial_s^2 \kappa + \kappa^2 \left( \tilde \kappa - \frac g L \right)
	=
	\partial_s^2 \kappa + \kappa^2 \left( \kappa - \frac { R + g } L \right)
	.
\]
For the cases of (AP) and (LP),
$ R + g > 0 $ as $ R > 0 $ and $ g \geqq 0 $.
In the case of (JP),
\[
	R + g
	=
	\frac { L^2 } A \geqq 0
	.
\]
$ \partial_s^2 \kappa \leqq 0 $ holds for $ s \in S_t $.
Hence,
we have
\[
	\partial_s^2 \kappa + \kappa^2 \left( \kappa - \frac { R + g } L \right)
	\leqq
	\kappa^3
	=
	K^3
\]
for $ s \in S_t $,
and
\[
	\frac { d^+ K } { dt } (t)
	\leqq
	\max_{ s \in S_t } \partial_t \kappa
	\leqq
	K^3 (t)
	.
\]
We calculate Dini's derivative of $ K^{-2} $ as
\begin{align*}
	\frac { d^+ } { dt } K^{-2} (t)
	= & \
	\limsup_{ h \to + 0 } \frac { K^{-2} ( t + h ) - K^{-2} (t) } h
	\\
	= & \
	\limsup_{ h \to + 0 }
	\frac { ( K ( t ) + K ( t + h ) )( K ( t ) - K ( t + h ) ) }
	{ K^2 ( t + h ) K^2 (t) h }
	\\
	= & \
	- 2 K^{-3} (t)
	\liminf_{ h \to + 0 } \frac { K( t + h ) - K(t) } h
	\\
	\geqq & \
	- 2 K^{-3} (t)
	\limsup_{ h \to + 0 } \frac { K( t + h ) - K(t) } h
	\\
	= & \
	- 2 K^{-3} (t)
	\frac { d^+ K } { dt } (t)
	\geqq
	- 2
	.
\end{align*}
According to the assumption of the theorem,
there exists a sequence $ \{ t_k \}_{ k \in \mathbb{N} } $ such that $ t_k \to T - 0 $ and $ K( t_k )^{-2} \to 0 $ as $ k \to \infty $.
Using \cite[Theorem 3]{HT},
we have
\[
	K^{-2} ( t_k ) - K^{-2} (t)
	\geqq
	\wideubar{\int}_t^{ t_k }
	\frac { d^+ } { dt } K^{-2} (t) \, dt
	\geqq
	- 2 ( t_k - t )
\]
for $ t_k \in ( t , T ) $.
Therefore,
we can conclude that
\[
	K^{-2} (t)
	\leqq 2 ( T - t )
\]
by $ k \to \infty $
\qed
\end{proof}
\begin{thm}
Let $ T \in ( 0 , \infty ) $ be the blow-up time for a solution of one of (AP),
(LP),
or (JP).
Assume that
\[
	\sup_{ t \in [ 0 , T ) } \max_{ s \in \mathbb{R} / L(t) \mathbb{Z} } \kappa ( s , t )
	< \infty
	.
\]
\par
For the solution of (AP),
\[
	\min_{ s \in \mathbb{R} / L(t) \mathbb{Z} } \kappa ( s , t )
	\leqq
	- \frac 1 { \sqrt { 4 ( T - t ) } }
\]
holds.
\par
For the solution of (LP),
\[
	\min_{ s \in \mathbb{R} / L(t) \mathbb{Z} } \kappa ( s , t )
	\leqq
	- \left\{ \frac { 2 \pi n } { 9 L(0) ( T - t ) } \right\}^{ \frac 13 }
\]
holds.
\par
For the solution of (JP),
there exists a time $ T_\ast \in [ 0 , T ) $ such that
\[
	- \min_{ s \in \mathbb{R} / L(t) \mathbb{Z} } \kappa ( s,t )
	\geqq
	\max_{ s \in \mathbb{R} / L(t) \mathbb{Z} } \kappa ( s,t )
\]
holds for $ t \in [ T_\ast , T ) $.
Additionally,
it holds that
\[
	\min_{ s \in \mathbb{R} / L(t) \mathbb{Z} } \kappa ( s , t )
	\leqq
	- \frac 1 { \sqrt { 2 C_\ast ( T - t ) } }
	,
\]
where
\[
	C_\ast
	=
	1 + \frac { L( T_\ast )^2 } { 4 \pi n A ( T_\ast ) } .
\]
\end{thm}
\begin{rem}
The time $ T_\ast $ above exists for all cases.
\end{rem}
\begin{proof}
Here,
we set
\begin{align*}
	K (t) = & \
	- \min_{ s \in \mathbb{R} / L(t) \mathbb{Z} } \kappa ( s , t )
	,
	\\
	\frac { d^+ K } { dt } (t)
	= & \
	\limsup_{ h \to + 0 } \frac { K( t + h ) - K(t) } h
	.
\end{align*}
Define the set $ S_t $ by $ S_t = \{ s \in \mathbb{R} / L(t) \mathbb{Z} \, | \, - \kappa ( s,t ) = K(t) \} $.
As shown before,
it holds that
\[
	\frac { d^+ K } { dt } ( t  )
	=
	\max_{ s \in S_t } \partial_t ( - \kappa ) .
\]
$ - \kappa $ satisfies
\[
	\partial_t ( - \kappa )
	=
	\partial_s^2 ( - \kappa )
	+
	( - \kappa )^2
	\left\{
	( - \kappa ) + \frac { R + g } L
	\right\}
	.
\]
Since $ \partial_s^2 ( - \kappa ) \leqq 0 $ and $ - \kappa = K $ for $ s \in S_t $,
\[
	\partial_t ( - \kappa )
	\leqq
	K^3 + \frac { ( R + g ) K^2 } L
	.
\]
If $ \kappa \leqq C < \infty $ holds on $ [ 0 , T ) $,
then,
\[
	L \max\{ C^2 + K^2 \} \geqq
	\int_0^L \kappa^2 ds = W \to \infty \mbox{ as } t \to T - 0
\]
by Theorem \ref{Theorem blow up W}.
Since $ L $ is bounded,
we conclude that $ K \to \infty $ as $ t \to T - 0 $.
Therefore,
$ | \kappa | \leqq \max\{ C , K \} \leqq K $ near $ T $.
Hence,
there exists $ T_\ast \in [ 0 , T ) $ as mentioned in the statement.
Considering $ t \geqq T_\ast $,
we may assume that $ | \kappa | \leqq K $.
\par
In the case of (AP),
since $ g = 0 $,
\[
	\frac { ( R + g ) K^2 } L
	=
	\frac { R K^2 } L .
\]
Using this and
\[
	R = \int_0^L \kappa \, ds
	\leqq \int_0^L | \kappa | \, ds \leqq LK ,
\]
we have $ \partial_t ( - \kappa ) \leqq 2 K^3 $ on $ S_t $,
{\it i.e.},
\[
	\frac { d^+ K } { dt } ( t  ) \leqq 2 K^3 .
\]
Consequently,
we obtain the assertion as before.
\par
In the case of (LP),
\[
	\frac { K^2 g } L
	=
	\frac { K^2 I_0 } { R L }
	=
	\frac { K^2 } R \int_0^L \tilde \kappa^2 ds 
	\leqq
	\frac { K^2 } R \int_0^L \kappa^2 ds
	\leqq
	\frac { L K^4 } R
	.
\]
The estimate $ \displaystyle{ \frac RL \leqq K } $ holds for all cases.
Hence,
\[
	K^3
	=
	\frac LR \cdot \frac RL \cdot K^3
	\leqq
	\frac { L K^4 } R
	,
	\quad
	\frac { K^2 R} L
	= \left( \frac RL \right)^2 \frac { L K^2 } R
	\leqq
	\frac { L K^4 } R .
\]
Consequently,
we have
\[
	\frac { d^+ K } { dt } ( t ) \leqq \frac { 3 L K^4 } R = \frac { 3 L(0) K^4 } { 2 \pi n } .
\]
Here,
we use $ L \equiv L(0) $.
The statement follows from the above,
as shown before.
\par
In the case of (JP),
using $ \displaystyle{ R + g = \frac { L^2 } { 2A } } $ and Lemma \ref{basic properties},
we have
\[
	\frac { K^2 ( R + g ) } L
	=
	\frac { K^2 L } { 2 A }
	=
	\frac { L^2 } { 2A } \cdot \frac RL \cdot \frac { K^2 } R
	\leqq
	\frac { L( T_\ast )^2 } { 2 A( T_\ast ) } \cdot \frac { K^3 } R
	=
	\frac { L( T_\ast )^2 K^3 } { 4 \pi n A( T_\ast ) }
	.
\]
Hence,
it holds that
\[
	\frac { d^+ K } { dt } ( t ) \leqq
	\left( 1 + \frac { L( T_\ast )^2 } { 4 \pi n A( T_\ast ) } \right) K^3 ,
\]
which leads to the required conclusion and ends the proof.
\qed
\end{proof}
\section{Convergence of global solutions}
\label{Global solutions}
\par
In this section,
we assume that $ \vecf $ is a classical global solution of one of
(AP),
(LP),
or (JP),
and that the initial curve satisfies $ A(0) > 0 $.
We prove that $ \mathrm{Im} \vecf $ converges to an $ n $-fold circle exponentially as $ t \to \infty $.
\begin{rem}
However,
this conclusion is meaningless if $ n $-fold circles are only global solutions.
At least,
in the case of (AP),
under suitable assumptions on the initial curve,
regarding symmetry and convexity,
solutions exist globally in time even if $ n > 1 $.
See \cite{WangKong}.
\end{rem}
\par
Firstly we prove the decay of $ I_{-1} $.
\begin{lem}
For the global solution above,
$ I_{-1} (t) $ fulfills
\[
	0 \leqq I_{-1} (t)
	\leqq
	\frac { L(0)^2 I_{-1} (0) }{ L(t)^2 } \exp \left( - \int_0^t \frac { 8 \pi^2 n } { L( \tau )^2 } \, d \tau \right)
	.
\]
In particular,
the estimate
\[
	0 \leqq I_{-1} (t)
	\leqq
	\frac { L(0)^2 I_{-1} (0) }{ 4 \pi n A(0) } \exp \left( - \frac { 8 \pi^2 n } { L(0)^2 } t \right)
\]
is satisfied with respect to the global solution for (AP);
the estimate
\[
	0 \leqq I_{-1} (t)
	\leqq
	I_{-1} (0) \exp \left( - \frac { 8 \pi^2 n } { L(0)^2 } t \right)
\]
for the global solution of (LP).
In the case of (JP),
setting $ \displaystyle{ \bar L = \sup_{ t \in [ 0 , \infty ) } L(t) } $,
we have $ \bar L < \infty $,
and
\[
	0 \leqq I_{-1} (t)
	\leqq
	\frac { L(0)^2 I_{-1} (0) }{ 4 \pi n A(0) } \exp \left( - \frac { 8 \pi^2 n } { \bar L^2 } t \right) .
\]
\label{decay_of_I_{-1}}
\end{lem}
\begin{proof}
For global solutions,
we know,
from Theorem \ref{Theorem blow-up},
that $ I_{-1} (t) \geqq 0 $ .
Hence,
we have
\be
	4 \pi^2 n I_{-1} (t) \leqq I_0 (t)
	\label{Wirtinger type +}
\ee
by Lemma \ref{Wirtinger type}.
Consequently,
\pref{ d L^2 I_{-1} / dt } becomes
\[
	\frac d { dt } \left( L^2 I_{-1} \right)
	+
	\frac { 8 \pi^2 n } { L^2 } \left( L^2 I_{-1} \right)
	\leqq
	0
	.
\]
Solving this differential inequality,
we obtain the first assertion.
\par
We use $ \sqrt{ 4 n \pi A(0) } \leqq L(t) \leqq L(0) $ for (AP),
and $ L(t) \equiv L(0) $ for (LP).
Then,
the second assertion follows for these two cases.
\par
Now,
we consider the case of (JP).
Integrating \pref{ d L^2 I_{-1} / dt },
we have
\[
	L^2 I_{-1} + 2 \int_0^t I_0 d \tau = L_0^2 I_{-1} (0) .
\]
$ \displaystyle{ \frac { L^2 } A } $ is uniformly positive and bounded by Proposition \ref{basic properties}. 
From this,
\pref{ d L^2 / dt } with $ \displaystyle{ g = \frac { L^2 I_{-1} } { 2A } } $ and \pref{Wirtinger type +},
we have
\[
	\frac { d L^2 } { dt } + 2 I_0
	=
	\frac { 2 \pi n L^2 } A I_{-1}
	\leqq
	\frac { L^2 } { 2 \pi A } I_0
	\leqq
	C I_0 .
\]
Integrating this,
we have
\[
	L^2 + 2 \int_0^t I_0 ( \tau ) \, d \tau
	\leqq
	L_0^2 + C \int_0^t I_0 ( \tau ) \, d \tau
	\leqq
	L_0^2 \left( 1 + C I_{-1} (0) \right) .
\]
Hence,
$ \bar L < \infty $.
It follows from \pref{ dA / dt } and $ \displaystyle{ g = \frac { L^2 I_{-1} } { 2A } \geqq 0 } $ that
\[
	\frac { d A^2 } { dt } = L^2 I_{-1} \geqq 0 .
\]
Therefore,
the lower bound $ L $ follows from $ L(t)^4 \geqq ( 4 \pi n A(t) )^2 \geqq ( 4 \pi n A(0) )^2 $.
Consequently,
we obtain the second assertion for (JP).
\qed
\end{proof}
\par
We denote the statement of Lemma \ref{decay_of_I_{-1}} as
\[
	I_{-1} (t) \leqq C e^{ - \lambda_{-1} t } .
\]
\begin{cor}
For the global solution above,
there exists $ L_\infty > 0 $ and $ A_\infty > 0 $ such that
\[
	| L - L_\infty |
	+
	| A - A_\infty |
	\leqq
	C e^{ - \lambda_{-1} t } .
\]
\end{cor}
\begin{proof}
In the case of (AP),
by Proposition \ref{basic properties},
we have $ \displaystyle{ \frac { dL } { dt } \leqq 0 } $.
Hence,
we conclude the convergence of $ \displaystyle{ \lim_{ t \to \infty } L(t) } $.
Set the limit value as $ L_\infty $.
Since $ A(t) \equiv A (0) $,
and since $ \displaystyle{ \lim_{ t \to \infty } I_{-1} (t) = 0 }$,
it holds that
\[
	L_\infty^2 = \lim_{ t \to \infty } 4 \pi n A(t)
	= 4 \pi n A(t) = 4 \pi n A(0) > 0
\]
and $ L_\infty \leqq L \leqq L(0) $.
Therefore,
\begin{align*}
	0
	\leqq & \
	L - L_\infty
	=
	\frac { L^2 - L_\infty^2 } { L + L_\infty }
	=
	\frac { L^2 - 4 \pi n A } { L + L_\infty }
	=
	\frac { L^2 I_{-1} } { L + L_\infty }
	\\
	\leqq & \
	\frac { L(0)^2 I_{-1} } { 2 L_\infty }
	=
	\frac { L(0)^2 I_{-1} } { 4 \sqrt{ \pi n A(0) } }
	\leqq
	C e^{ - \lambda_{-1} t }
	.
\end{align*}
\par
In the case of (LP),
since $ \displaystyle{ \frac { dA } { dt } \geqq 0 } $ and since $ 4 \pi A \leqq L^2 = L(0)^2 $,
we conclude the convergence of $ \displaystyle{ \lim_{ t \to \infty } A(t) } $.
Set the limit value as $ A_\infty $.
Since $ L(t) \equiv L(0) $,
and $ \displaystyle{ \lim_{ t \to \infty } I_{-1} (t) = 0 } $,
it holds that $ 4 \pi n A_\infty = L(0)^2 $.
Consequently,
\pref{ dA / dt } with $ \displaystyle{ g = \frac { I_0 } { 2 \pi n } } $ yields
\[
	0
	\leqq
	A_\infty - A
	=
	\int_t^\infty \frac { I_0 } { 2 \pi n } \, dt
	=
	\frac { L_0^2 } { 4 \pi n } I_{-1} (t)
	\\
	\leqq
	C e^{ - \lambda_{-1} t } .
\]
Here,
we use \pref{ d L^2 I_{-1} / dt } and Lemma \ref{decay_of_I_{-1}}.
\par
In the case of (JP),
$ \displaystyle{ \frac { dA } { dt } = \frac { L^2 I_{-1} } { 2A } \geqq 0 } $.
By Proposition \ref{basic properties},
$ \displaystyle{ \frac A { L^2 } } $ is uniformly positive and bounded.
Combining the above two statements with Lemma \ref{decay_of_I_{-1}},
we conclude
\[
	0
	\leqq
	A_\infty - A
	=
	\int_t^\infty \frac { L^2 I_{-1} } { 2A } \, dt
	\leqq
	C \int_t^\infty I_{-1} dt
	\leqq
	C e^{ - \lambda_{-1} t }
	.
\]
Furthermore,
we estimate that
\begin{align*}
	| L - L_\infty |
	= & \
	\frac { | L^2 - L_\infty^2 | } { L + L_\infty }
	=
	\frac { | L^2 I_{-1} + 4 \pi n A -  4 \pi n A_\infty | } { L + L_\infty }
	\\
	\leqq & \
	\frac { L^2 I_{-1} + 4 \pi n | A - A_\infty | } { L_\infty }
	\leqq
	C e^{ - \lambda_{-1} t }
	.
\end{align*}
\qed
\end{proof}
\begin{cor}
For the global solution above,
it holds that
\[
	\int_t^\infty I_0 dt 
	\leqq C e^{ - \lambda_{-1} t } .
\]
\label{decay_of_int_t^infty I_0}
\end{cor}
\begin{proof}
We know that $ L $ is uniformly bounded dor all cases.
Therefore,
\pref{ d L^2 I_{-1} / dt } implies that
\[
	\int_t^\infty I_0 \, dt
	=
	\frac { L^2 I_{-1} } 2
	\leqq C e^{ - \lambda_{-1} t } .
\]
\qed
\end{proof}
\begin{lem}
For the global solution above,
there exists $ \lambda_0 > 0 $ such that
\[
	I_0 
	\leqq C e^{ - \lambda_0 t } .
\]
\label{decay_of_I_0}
\end{lem}
\begin{proof}
As in Section \ref{Blow-up},
we set
\[
	W = \int_0^L \kappa^2 ds
	,
	\quad
	R = \int_0^L \kappa \, ds
	,
	\quad
	J_p = L^{p-1} \int_0^L \tilde \kappa^p ds .
\]
As we know that $ L \to L_\infty > 0 $ as $ t \to \infty $,
it is enough to show that
\[
	L^2 I_0 
	\leqq C e^{ - \lambda_0 t } .
\]
Since $ I_0 = J_2 = LW - R^2 $,
we have from \pref{ d L^2 / dt } and Lemma \ref{Lemma 2nd}
\begin{align*}
	\frac d { dt } \left( L^2 I_0 \right)
	= & \
	\frac d { dt } \left( L^3 W - R^2 L^2 \right)
	=
	L^3 \frac { dW } { dt }
	+
	\left( \frac 32 LW - R^2 \right) \frac { d L^2 } { dt }
	\\
	= & \
	- 2 I_1 + J_4 + ( 3R - g ) J_3 + 3R ( R - g ) J_2 - R^3 g
	\\
	& \quad
	+ \,
	\left( \frac 32 I_0 + \frac 12 R^2 \right) ( - 2 I_0 + 2 Rg )
	\\
	= & \
	- 2 I_1 - 3 I_0^2
	+
	J_4 + ( 3R - g ) J_3 + 2 R^2 J_2 .
\end{align*}
We obtain
\[
	\frac d { dt } \left( L^2 I_0 \right)
	+ I_1
	+ 3 I_0^2
	\leqq
	C \left(
	I_0^3 + I_0 + I_0^{ \frac 53 } + | g |^{ \frac 43 } I_0^{ \frac 53 }
	\right)
\]
in a manner similar to the proof of Lemma \ref{Lemma 3rd}.
\par
Since $ g = 0 $ in (AP),
and $ \displaystyle{ g = \frac { L^2 I_{-1} } { 2A } } $ in (JP),
$ | g | $ is uniformly bounded for these cases.	
In (LP),
$ g = R^{-1} I_0 $.
Hence,
it holds for every case that
\[
	\frac d { dt } \left( L^2 I_0 \right)
	+ I_1
	+ 3 I_0^2
	\leqq
	C \left( I_0 + I_0^3 \right)
	.
\]
This can be presented as
\[
	\frac d { dt } \left( L^2 I_0 \right)
	+ I_1
	+ I_0^2 \left( 3 - C I_0 \right) 
	\leqq
	C I_0
	.
\]
By Corollary \ref{decay_of_int_t^infty I_0},
there exists $ t_0 > 0 $ such that
\[
	I_0 (t_0) \leqq \frac 1C
	,
	\quad
	\int_{ t_0 }^\infty I_0 dt \leqq \frac { L^2 } C
	.
\]
Set
\[
	t_1 = \sup \left\{ t \in [ t_0 , \infty ) \, \left| \,
	I_0 (t) < \frac 3C \quad ( t \in [ t_0 , \infty ) ) \right. \right\}
	.
\]
If $ t_1 < \infty $,
then,
\[
	\limsup_{ t \to t_1 - 0 } I_0 ( t ) = \frac 3C < \infty .
\]
For $ t \in ( t_0 , t_1 ) $,
we have
\[
	\frac d { dt } \left( L^2 I_0 \right)
	\leqq
	C I_0
	,
\]
and therefore,
\[
	I_0 (t)
	\leqq
	I_0 ( t_0 ) + \frac 1 { L^2 } \int_{ t_0 }^t I_0 dt
	\leqq
	\frac 2C
	=
	\frac 23 \limsup_{ t \to t_1 } I_0 ( t )
	.
\]
Letting $ t \to t_1 - 0 $,
we obtain a contradiction.
Consequently,
$ t_1 = \infty $,
that is,
$ \displaystyle{ I_0 (t) < \frac 3C } $ for $ t \in [ t_0 , \infty ) $.
Since we know that $ I_0 $ is uniformly bounded,
we obtain
\[
	\frac d { dt } \left( L^2 I_0 \right)
	+
	I_1
	+
	3 I_0^2
	\leqq
	C I_0
	.
\]
It follows from Wirtinger's inequality and the uniform estimate of $ L^2 $ that
\[
	\frac d { dt } \left( L^2 I_0 \right)
	+
	2 \lambda L^2 I_0
	\leqq
	C I_0
\]
for some constant $ \lambda > 0 $.
Multiplying both sides by $ e^{ 2 \lambda t } $,
and integrating from $ \displaystyle{ \frac t2 } $ to $ t $,
we have
\begin{align*}
	e^{ 2 \lambda t } L(t)^2 I_0 (t)
	\leqq & \
	C e^{ \lambda t } L \left( \frac t 2 \right)^2 I_0 \left( \frac t2 \right)
	+
	C \int_{ \frac t2 }^t e^{ 2 \lambda \tau } I_0 ( \tau ) \, d \tau
	\\
	\leqq & \
	C e^{ \lambda t }
	+
	C e^{ 2 \lambda t }\int_{ \frac t2 }^\infty I_0 ( \tau ) \, d \tau
	.
\end{align*}
That is,
we have
\[
	L(t)^2 I_0 (t)
	\leqq
	C e^{ - \lambda t }
	+
	C \int_{ \frac t2 }^\infty I_0 ( \tau ) \, d \tau
	.
\]
Using the uniform estimate of $ L $ and the exponential decay of $ \displaystyle{ \int_{ \frac t2 }^\infty I_0 dt } $,
we finally obtain the exponential decay of $ I_0 $.
\qed
\end{proof}
\begin{cor}
We have $ \widetilde I_{-1} \leqq C e^{ - \lambda_0 t } $.
\end{cor}
\begin{proof}
The assertion follows from Lemmas \ref{Wirtinger type tilde} and \ref{decay_of_I_0}.
\qed
\end{proof}
\par
Once we obtain the exponential decay of $ \widetilde I_{-1} $ and $ I_0 $,
we can obtain the convergence of $ \mathrm{Im} \vecf $ to an $ n $-fold circle as $ t \to \infty $.
\begin{thm}
Let $ \vecf $ be a classical global solution of one of
(AP),
(LP),
or (JP),
with the smooth initial curve satisfying $ A(0) > 0 $.
Then,
$ \mathrm{Im} \vecf $ converges to an $ n $-fold circle with centre $ \vecc_\infty $,
and radius $ \displaystyle{ r_\infty = \frac { L_\infty } { 2 \pi n } } $ in the following sense.
Set
\begin{gather*}
	\vecf ( s, t )
	=
	\vecc (t) + r(t)
	\left(
	\cos \frac { 2 \pi n ( s + \sigma (t) ) } { L(t) } ,
	\sin \frac { 2 \pi n ( s + \sigma (t) ) } { L(t) }
	\right)
	+ \vecrho ( s , t ) ,
	\\
	\vecc (t)
	=
	\frac 1 { L(t) } \int_0^{ L(t) } \vecf ( s,t ) \, ds
	,
	\quad
	r(t) = \frac { L(t) } { 2 \pi n } ,
\end{gather*}
with the $ \mathbb{R} / L(t) \mathbb{Z} $-valued function $ \sigma $ defined by
\[
	\hat f(n) (t) = \sqrt{ L(t) } r(t) \exp \left( \frac { 2 \pi i n \sigma (t) } { L(t) } \right) .
\]
Then,
there exist $ \vecc_\infty \in \mathbb{R}^2 $,
$ \displaystyle{ r_\infty = \frac { L_\infty } { 2 \pi n } > 0 } $,
$ \sigma_\infty \in \mathbb{R} / L_\infty \mathbb{Z} $,
$ \lambda > 0 $,
and $ C > 0 $ such that
\[
	\| \vecc (t) - \vecc_\infty \|
	+
	| r(t) - r_\infty |
	+
	\left| \frac { \sigma (t) } { L(t) } - \frac { \sigma_\infty } { L_\infty } \right|
	\leqq
	C e^{ - \lambda t } .
\]
Furthermore,
for $ k \in \{ 0 \} \cup \mathbb{N} $,
there exist $ \gamma_k > 0 $ and $ C_k > 0 $ such that
\[
	\| \vecrho ( \cdot , t ) \|_{ C^k ( \mathbb{R} / L(t) \mathbb{Z} ) }
	\leqq
	C_k e^{ - \gamma_k t } .
\]
\label{Theorem convergence}
\end{thm}
\par
To prove this theorem we use the standard energy method with help of Theorem \ref{Theorem interpolation}.
Since the argument is almost parallel to those of \cite{Nagasawa-Nakamura1,Nakamura2},
we omit the details.


\begin{thebibliography}{99}
\bibitem{CLN}
	Chow, B.,
	P. Lu,
	\& L. Ni,
	``Hamilton's Ricci Flow'',
	Graduate Studies in Mathematics {\bf 77},
	American Mathematical Society,
	Providence,
	Science Press,
	Beijing,
	New~York,
	2006.
\bibitem{G}
	Gage, M.,
	{\it On an area-preserving evolution equation for plane curves},
	in ``Nonlinear problems in geometry (Mobile, Ala., 1985)'',
	Contemp.\ Math.\ {\bf 51},
	Amer. Math. Soc.,
	Providence,
	1986,
	pp.51--62.
\bibitem{HT}
	Hagood, J. W. \& B. S. Thomson,
	{\it Recovering a function from a Dini derivative},
	Amer.\ Math.\ Monthly {\bf 113} (1) (2006),
	34--46. 
\bibitem{JP}
	Jiang, L. \& S. Pan,
	{\it On a non-local curve evolution problem in the plane},
	Comm.\ Anal.\ Geom.\ {\bf 16} (1) (2008),
	1--26.
\bibitem{MZ}
	Ma, L. \& A. Zhu,
	{\it On a length preserving curve flow},
	Monatsh.\ Math.\ {\bf 165} (1) (2012),
	57--78.
\bibitem{Nagasawa-Nakamura1}
	Nagasawa, T. \& K. Nakamura,
	{\it Interpolation inequalities between the deviation of curvature and the isoperimetric ratio with applications to geometric flows},
	Adv.\ in Differential Equations {\bf 24} (2019) (9--10),
	581--608.
\bibitem{Nakamura2}
	Nakamura, K.,
	{\it An application of interpolation inequalities between the deviation of curvature and the isoperimetric ratio to the length-preserving flow},
	arXiv:1811.11576.
\bibitem{WangKong}
	Wang, X.-L.\ \& L.-H. Kong,
	{\it Area-preserving evolution of nonsimple symmetric plane curves},
	J. Evol.\ Equ.\ {\bf 14} (2) (2014),
	387--401.
\end{thebibliography}
\end{document}